\documentclass[article,12pt]{amsart}
\usepackage{amsfonts}
\usepackage{amsmath}
\usepackage{amssymb}
\usepackage{graphicx}
\usepackage{epsfig}
\usepackage{enumerate}
\usepackage{fullpage}
\usepackage{color}
\usepackage{bbm}

\allowdisplaybreaks

\numberwithin{equation}{section}

\theoremstyle{plain}
\newtheorem{theo}{Theorem}[section]
\newtheorem{prop}{Proposition}[section]

\newtheorem{lem}{Lemma}[section]

\newtheorem{defi}{Definition}[section]
\newtheorem{rem}{Remark}[section]

\newcommand{\dR}{\mathbb{R}}

\email{marius.soltane.etu@univ-lemans.fr}

\keywords{Asymptotic efficiency, Autoregressive process, AUMPI test, Ergodic control, Gaussian noise, LAN property, Maximum Likelihood estimator.}

\begin{document}

\title[Asymptotic efficiency for the MLE]
{Asymptotic efficiency in the autoregressive process driven by a stationary gaussian noise
\vspace{2ex}}
\author[M. Soltane]{Marius Soltane}
\address{Laboratoire Manceau de Math\'ematiques, Le Mans Universit\'e, Avenue O. Messiaen, 72085 Le Mans cedex 9, France.}

\begin{abstract}
The first purpose of this article is to obtain $a.s.$ asymptotic properties of the maximum likelihood estimator in the autoregressive process driven by a stationary Gaussian noise. The second purpose is to show the local asymptotic normality property of the likelihoods ratio in order to get a notion of asymptotic efficiency and to build an asymptotically uniformly invariant most powerful procedure for testing the significance of the autoregressive parameter.  
\end{abstract}

\maketitle

\section{Introduction.}\label{Sec1}
\label{SecIntro}
Classical autoregressive processes driven by strong white noise were introduced by Box-Jenkins and studied as early in \cite{BD1991}. Now models using autoregressive processes with dependant perturbations are widely used in various fields, especially in econometrics and finance. The asymptotic behavior of the least square estimator (LSE) is generally degraded for this type of process and no consistent for the autoregressive parameter (see \cite{BF2013} for an illustration of this fact where the author consider an AR(1) process driven by an AR(1) noise).  A more general study with an AR($p$) process driven by an AR(1) noise was realized in \cite{PROIA201377} and some asymptotic properties of the maximum likelihood estimator (MLE) in the model presented later was studied in \cite{BCK14}. In this study no attention is paid to the obtention of $a.s.$ properties of the estimation and no rate of convergence is obtained.
We also address the questions of the asymptotic efficiency for the MLE and the optimality of the test of significance of the parameter driving the autoregressive dynamics.

\

We consider in this paper the stochastic process $(X_n)$ indexed on $\mathbb{N}$ and satisfying for all $n \in \mathbb{N}$,
\begin{equation}\label{Equa_Auto_P1}
X_n = \sum_{i=1}^p \theta_i X_{n-i} + \xi_n.
\end{equation}
In (\ref{Equa_Auto_P1}) the nuisance process $(\xi_n)$ is a stationary centered Gaussian process and we assume that $X_{-p} = \dots = X_{-1} = 0.$

\

To obtain an explicit formula for the MLE, a transformation of the model  is carried out, in order to obtain an independent noise.
Typically, the arguments used to obtain the asymptotic properties of the  estimators in this type of processes call for results on martingales.
To apply these results, ergodicity arguments are invoked but we will see later that this can not be verified, strictly speaking. We will therefore present in section \ref{Sec4} a new method to apply the standard results for martingales and to obtain the desired properties.
The second section is devoted to the presentation of the model in particular, we recall the well-know results related to this AR process. The third section contains the presentation of the results etablished on the MLE. In particular, we are getting $a.s.$ convergence for the filtered process, which furthermore makes it possible to obtain quadratic strong law for the MLE and its the strong consistency. In the same section, we also get the LAN property which will allow us to build an optimal test. 
For the sake of clarity, a part containing the auxiliary results precedes the part containing the proofs of the mains results

\section{Preliminaries.}\label{Sec1pre}

\subsection{Model and assumptions.}

\

\

In the rest of the article $\Vert x \Vert$ refers to the euclidian norm of a vector $x$ and $Id_p$ is the identity matrix of size $p \times p$. When $M$ is a matrix, $\Vert M \Vert$ is the usual matrix norm induced by the Euclidian norm. Finally, $A^*$ is the transpose of $A.$

\

\noindent We use (\ref{Equa_Auto_P1}) in order to write the model in a vectorial form. Let 

\begin{align*}
A_0 = \begin{pmatrix}
\theta_1 & \theta_2 & \dots & \dots & \theta_p \\
1 & 0 & \dots & 0 & 0 \\
0 & 1 & \dots & 0 & 0 \\
\dots & \dots & \dots & \dots & \dots \\
0 & \dots & 0 & 1 & 0
\end{pmatrix} \ \text{and} \ Y_n=(X_n, X_{n-1}, \dots, X_{n-p+1})^*.
\end{align*}

\noindent Then, for all $n$,

\begin{equation}\label{AR(p)_vect_P1}
Y_n = A_0Y_{n-1} + b\xi_n,
\end{equation}

\noindent where $b=(1,0_{1 \times (p-1)})^*.$
In all that follows, we retain the following hypotheses :

\begin{itemize}
\item $(H_1)$ $\rho(A_0) < 1$ where $\rho$ refers to the spectral radius of the matrix $A_0$. The Parametric space is therefore $\Theta = \{ \theta \in \mathbb{R}^p | \rho(A_0) < 1\}.$
\item $(H_2)$ the covariance fonction $r$ of the nuisance process satisfies $r(n) = O(\frac{c}{n^{\alpha}})$ when $n \rightarrow \infty$. In this relation, $\alpha > 0$ and $c$ is a positive constant. 
\item $(H_3)$ Let $(\beta_n)$ be the PACF of $(\xi_n)$, we suppose that $\beta_n^2=O(\frac{1}{n^{\alpha}})$ for some $\alpha>1.$
\end{itemize}

\noindent The last assumption is slightly stronger than $(\beta_n^2) \in \ell^1 (\mathbb{N})$ which holds in this study, but it will be required in our technical proofs.
Let $f_{\xi}$ the spectral density of the process $(\xi_n)$, not thats $(H_2)$ can be rewritten (see \cite{BCK14}) in term on condition of the spectral density $f_{\xi}$ as 

\begin{equation}\label{condition_DS}
\displaystyle \int_{-\pi}^{\pi} \vert \log f_{\xi}(\lambda) \vert \, \mathrm{d}\lambda < \infty.
\end{equation}

\subsection{Model Transformation.}\label{ModelTrans}

\

\

In this section, we present a linear transformation in order to obtain a Markov process driven by independent noise.
Let $\sigma_1\varepsilon_1 = \xi_1$ and for all $n \geqslant 2$,

\begin{equation}\label{cond_exp1_nuissanceP1}
\sigma_n\varepsilon_n = \xi_n - \mathbb{E}(\xi_n | \xi_1, \dots, \xi_{n-1}),
\end{equation}

\noindent where $(\varepsilon_n)$ are $i.i.d.$ and $\varepsilon_n \sim \mathcal{N}(0,1).$ By the Theorem of Normal Correlation (Theorem 13.1 in \cite{LS01}) we have,

\begin{equation}\label{Esp_cond_Filt_P1}
\sigma_n\varepsilon_n = \sum_{i=1}^n k(n,i)\xi_i,
\end{equation}

\noindent where $(k(n,i)_{\{1 \leqslant i \leqslant n, n \in \mathbb{N}^*\}})$ is a deterministic kernel and $(\sigma_n^2)$ is the variance of innovations.

\noindent Let
\begin{equation}\label{coor_part_P1}
\beta_{n-1} = -k(n,1).
\end{equation}

\noindent By the Durbin-Levinson algorithm (see \cite{D60}), the following relations are true and make it possible to calculate the coefficients interventing in (\ref{Esp_cond_Filt_P1}).

\begin{equation}\label{Durbin_var_P1}
\sigma_n^2=\prod_{i=1}^{n-1} (1-\beta_i^2), \ n \geqslant 2, \ \sigma_1=1,
\end{equation}

\begin{equation}\label{Durbin_beta_P1}
\sum_{i=1}^n k(n,i)r(i) = \beta_n\sigma_n^2, \ k(n,n) = 1,
\end{equation}

\begin{equation}\label{Durbin_coeff_P1}
k(n+1,n+1-i)=k(n,n-i)-\beta_nk(n,i), \ 1 \leqslant i \leqslant n-1.
\end{equation}

\noindent Now, let also

\begin{equation}\label{Trans1_P1}
Z_n= \sum_{i=1}^n k(n,i)Y_i,
\end{equation}

\noindent and,

\begin{equation}\label{Process_Filtré_P1}
\zeta_n= \begin{pmatrix}
Z_n \\
\sum_{k=1}^{n-1} \beta_kZ_k
\end{pmatrix}.
\end{equation}

\noindent 
The initial estimation problem of $\theta$ is replaced by the estimation of the unknown parameter $\theta$ from the observations $\zeta = (\zeta_n, \ n \geqslant 1).$
It was shown in \cite{BK12} that $(\zeta_n)$ is a $2p$-dimensionnal Markov process. More precisely, for all $n \in \mathbb{N}^*$, 

\begin{equation}\label{nouveau_zeta_P1}
\zeta_n=\widetilde{A}_{n-1}\zeta_{n-1} + \ell\sigma_n\varepsilon_n,
\end{equation}

\noindent where 

\begin{align*}
\widetilde{A}_n = \begin{pmatrix}
A_0 & \beta_nA_0 \\
\beta_nId_p & Id_p
\end{pmatrix}, \ \ell=(1,0_{1 \times (2p - 1)})^* \ \text{and} \ \zeta_0 = 0_{2p \times 1}.
\end{align*}

\

\noindent Therefore, the log-likelihood function is given by

\begin{equation}\label{log_vraisemblance_P1}
\log \mathcal{L}(\theta,X^{(n)}) = -\frac{1}{2} \sum_{i=1}^n \left(\frac{\ell^*(\zeta_i-\widetilde{A}_{i-1}\zeta_{i-1})}{\sigma_i}\right)^2 - \frac{n}{2} \log 2\pi - \frac{1}{2}\sum_{i=1}^n\sigma_i^2
\end{equation}

\noindent where $\theta=(\theta_1, \dots, \theta_p)$ and $X^{(n)} = (X_0,X_1, \dots,X_n).$

\subsection{Construction of the MLE and reminders of known properties.}

\

\

Using (\ref{log_vraisemblance_P1}), it follows that the MLE is given by

\begin{equation}
\widehat{\theta}_n=\left(\sum_{i=1}^n \frac{a_{i-1}^*\zeta_{i-1}\zeta_{i-1}^*a_{i-1}}{\sigma_i^2}\right)^{-1}\left(\sum_{i=1}^n \frac{a_{i-1}^*\zeta_{i-1}\ell^*\zeta_i}{\sigma_i^2}\right),
\end{equation}

\noindent where $a_n = (Id_p ,  \beta_nId_p)^*.$
The matrix $I(\theta)$ is the unique solution of the Lyapunov equation given by

\begin{equation}\label{Lyapu_P1}
I(\theta) = A_0^*I(\theta)A_0+bb^*,
\end{equation}

\noindent and we have the following properties (see Theorem 1 in \cite{BCK14}) :

\begin{equation}\label{convergence_est_P1}
{\widehat{\theta}_n}\xrightarrow[n\to\infty]{\mathbb{P}_{\theta}^{(n)}}{\theta},
\end{equation}

\begin{equation}\label{loi_est1_P1}
{\sqrt{n}(\widehat{\theta}_n-\theta)}\xrightarrow[n\to\infty]{\mathcal{L}}{\mathcal{N}(0,I(\theta)^{-1})}
\end{equation}

\noindent under $\mathbb{P}_{\theta}^{(n)}.$

\section{Mains Results.}\label{Sec2}

\subsection{Almost sure properties of the MLE}\label{Sec2.1}

\

\

The results of this part is the strong consistency of the MLE, the quadratic strong law for the MLE and a law of the iteraded logarithm.
All the results presented in this section are valid under $(H_1), (H_2), (H_3).$

\begin{theo}\label{Theo1_P2}
The MLE is strongly consistent, i.e.
\begin{equation}\label{strong_MLE}
{\widehat{\theta}_n}\xrightarrow[n\to\infty]{a.s.}{\theta}.
\end{equation}
\end{theo}

\begin{proof}
See Section \ref{Preuve_2}.
\end{proof}

\begin{theo}\label{Theo2_LFQ_P2}
We have the following quadratic strong law for the MLE,
\begin{equation}\label{LFQ_MLE_P2}
{\frac{1}{\log n}\sum_{k=1}^n(\widehat{\theta}_k-\theta)(\widehat{\theta}_k-\theta)^*}\xrightarrow[n\to\infty]{a.s.}{I(\theta)^{-1}}.
\end{equation}
The limit above is the same as the asymptotic covariance matrix in (\ref{loi_est1_P1}) and $I(\theta)$ is defined in (\ref{Lyapu_P1}).
\end{theo}
\begin{proof}
See Section \ref{Preuve_3}.
\end{proof}

To conclude this section, we give the LLI of the MLE and hence the convergence rate of the MLE.

\begin{prop}\label{Prop_LLI}
We have the following properties for all $v \in \dR^p$,
\begin{align*}
\limsup\limits_{n \to \infty} \left(\frac{n}{2\log \log n} \right)^{\frac{1}{2}} v^*(\widehat{\theta}_n-\theta) & =  - \liminf\limits_{n \to \infty} \left( \frac{n}{2\log \log n} \right)^{\frac{1}{2}} v^*(\widehat{\theta}_n-\theta) \\
& = (v^*I(\theta)^{-1}v)^{\frac{1}{2}} \ a.s.
\end{align*}
Consequently,
\begin{equation}\label{vitesse_MLE}
\left\Vert \widehat{\theta}_n - \theta \right\Vert^2 = O\left( \frac{\log \log n}{n} \right) \ a.s.
\end{equation}
\end{prop}

\begin{proof}
See Section \ref{Preuve_3.1}.
\end{proof}

\subsection{Local asymptotic normality property and application.}\label{Sec2.2}

\

\

The LAN (local asymptotic normality) property is an important notion under which we can define a notion of asymptotic efficiency for estimators (see \cite{LeCam1990}). Before stating the results, we remind for the reader's convenience some properties and definitions under LAN statistical experiments. The LAN property for stationary Gaussian process was obtained in \cite{CGL13} with conditions on the spectral density. We present here direct computation based on the particular autoregressive structure in order to obtain the LAN property.

\begin{defi}\label{LAN_property_P3}
We will say that a familly of measures $\mathbb{P}_{\theta}^{(n)}$ is LAN in $\theta_0 \in M \subset \mathbb{R}^d$ if the following conditions are satisfied about the likelihood ratio,

\begin{equation}\label{LAN_likelihood}
L_n(u) = \frac{d\mathbb{P}_{\theta_0+\phi_n (\theta_0)u}^{(n)}}{d\mathbb{P}_{\theta_0}^{(n)}},
\end{equation}

\begin{equation}\label{LAN_repre_P3}
L_n(u)=exp\left(\langle u,Z_n(\theta_0) \rangle - \frac{1}{2}\langle u, J(\theta_0)u \rangle + R_n(\theta_0,u)\right),
\end{equation}

\noindent where

\begin{equation}\label{score_LAN}
{Z_n(\theta_0)}\xrightarrow[n\to\infty]{\mathcal{L}}{\mathcal{N}(0,J(\theta_0))},
\end{equation}

\noindent and,

\begin{equation}\label{reste_LAN}
{R_n(\theta_0,u)}\xrightarrow[n\to\infty]{\mathcal{L}}0
\end{equation}

\noindent under $\mathbb{P}_{\theta_0}^{(n)}.$

\

\noindent The sequence $(\phi_n(\theta_0))$ satisfied

\begin{equation}\label{LAN_matrix_rate}
{\phi_n(\theta_0)}\xrightarrow[n\to\infty]{}0.
\end{equation}

\noindent In this definition $u \in K \subset \mathbb{R}^d$, $\phi_n (\theta_0)$ are non-singular matrix rate and $J(\theta_0)$ is a non singular $d \times d$ matrix.
\end{defi}

\begin{theo}\label{theo_inf_P3}
Suppose that the family of measures $\left\lbrace \mathbb{P}_{\theta}^{(n)},\theta \in M \subset \mathbb{R}^d \right\rbrace$ is LAN in $\theta_0$. Then for any $\delta >0$, 

\begin{equation}\label{borne_LAN_P3}
\liminf\limits_{n \rightarrow \infty} \  \sup\limits_{ \parallel \phi_n(\theta_0)^{-1}(\theta-\theta_0) \parallel \leqslant \delta} \ 
\mathbb{E}_{\theta_0}^{(n)}\left
 (f\left(\phi_n(\theta_0)^{-1}(\widehat{\theta}_n-\theta)\right)   \right) \geqslant 
{ \int_{\mathbb{R}^m}^{} {f\left(J(\theta_0)^{-\frac{1}{2}}x\right)\Phi_d(x)} \, \mathrm{d}x},
\end{equation}

\noindent for any estimator $\widehat{\theta}_n$ and for any cost function $f$ such that $f$ is continuous, symmetric, quasi-convex and ${f(z)exp(-\frac{\parallel z \parallel^2}{2})}\rightarrow 0$ when $\Vert z \Vert \rightarrow \infty.$ Here $\Phi_d$ is the density of the standard $d$-dimensionnal Gaussian distribution. 
\end{theo}

\noindent For a proof of the last result see Theorem 12.1, chapter 2 in \cite{IK81}.
We can now give the LAN property in the model that interests us.

\begin{theo}\label{Theo_LAN_AR(p)_P3}
With the notation of Definition \ref{LAN_property_P3} and denoting $\phi_n (\theta_0)=\frac{1}{\sqrt{n}}Id_p$ we have
\begin{equation}\label{LAN_AR(p)_P3}
\log \left( \frac{\mathcal{L}(\theta_0+\phi_n (\theta_0)u,X^{(n)})}{\mathcal{L}(\theta_0,X^{(n)})} \right) = \langle u, \frac{M_n}{\sqrt{n}}\rangle - \frac{1}{2}\langle u,I(\theta_0)u \rangle + R_n(\theta_0,u),
\end{equation}
where $(\frac{M_n}{\sqrt{n}})$ satisfied condition (\ref{score_LAN}) under $\mathbb{P}_{\theta_0}^{(n)}$ with $J(\theta_0)=I(\theta_0)$ and $R_n(\theta_0,u)$ satisfied condition (\ref{reste_LAN}) under $\mathbb{P}_{\theta_0}^{(n)}.$
In this Theorem, $u \in B(0;R)$ for any $R>0$.
\end{theo}
\begin{proof}
See Section \ref{Preuve_4}.
\end{proof}

\noindent We are now in position to give a result concerning the asymptotic efficiency of the MLE.

\begin{prop}\label{eff_MLE_P3}
Under $(H_1)$ and $(H_2),$ the MLE is asymptotically efficient, more precicely the lower-bound given by the Theorem \ref{theo_inf_P3} is reached for the MLE.
\end{prop}
\begin{proof}
See Section \ref{Preuve_5}.
\end{proof}

\noindent We will now focus on the optimality of the multidimensional hypotheses test in the autoregressive setting. Always for reader's convenience we recall notions and results on the tests which were introduced in \cite{CHA96}. Suppose that the familly of measures $\mathbb{P}_{\theta}^{(n)}$ is LAN in $\theta_0$. We would like to build an optimal procedure to test $\theta = \theta_0$ against $\theta \neq \theta_0.$

\begin{defi}\label{def_AUMP_P3} A test $\phi_n^1$ is said $AUMP(\alpha)$ (asymptotically uniformly most powerful of level $\alpha$) if
\begin{equation}\label{ineg_AUMP1_P3}
\limsup\limits_{n \to \infty}  \ \mathbb{E}_{\theta_0}^{(n)}(\phi_n^1) \leqslant \alpha,
\end{equation}
and for any other test $\phi_n^2$ of asymptotic level $\alpha$,  
\begin{equation}\label{ineg_AUMP2_P3}
\limsup\limits_{n \to \infty} \ \mathbb{E}_{\theta_0+\phi_n(\theta_0)u}^{(n)} (\phi_n^2) \leqslant \liminf\limits_{n \to \infty}  \ \mathbb{E}_{\theta_0+\phi_n(\theta_0)u}^{(n)} (\phi_n^1).
\end{equation}
\end{defi}

\begin{rem}\label{Rem4}
We give a lemma set in \cite{CHA96} in order to formalize the next  definition. We formulate this lemma in our context, i.e. without the parameters of nuisance since in our case, it is possible to compute them ($via$ the Durbin-Levinson algorithm).
\end{rem}

\begin{lem}\label{lim_test_P3} With the notation of Definition \ref{LAN_property_P3}, for every test $\phi_n^1$ and every subsequence $n'$, we cand find a subsequence $n''$ of $n'$ and a test $\phi$ from $\mathbb{R}^p$ to $[0;1]$ such that for every $u \in K$,
\begin{equation}\label{lim2_test_P3}
\lim\limits_{n'' \to \infty} \ \mathbb{E}_{\theta_0+\phi_{n''}  (\theta_0)u}^{(n'')} (\phi_{n''}^1) = {\int_{\mathbb{R}^d}^{} {\phi(x)\Phi_d(x-J(\theta_0)u)}\, \mathrm{d}x },
\end{equation}
where $\Phi_d$ is defined as in Theorem \ref{theo_inf_P3}.
\end{lem}

\noindent We will now introduce an invariance principle by rotation who is involved in the next Definition.

\begin{defi}\label{AUMPI_test_P3}
A test $\phi_n^1$ is AUMPI($\alpha$) is the condition of the Definition \ref{def_AUMP_P3} are satisfied and for all subsequence $n'$ the corresponding test $\phi$ (obtained via Lemma \ref{lim_test_P3}) satisfied $\phi(Ru)=\phi(u)$ for any rotation from $\mathbb{R}^d$ to $\mathbb{R}^d$.
\end{defi}

To finish this section we give an AUMPI test to test the significance of the autoregressive parameter.

\begin{theo}\label{Theo_testAUMPI_AR(p)_P3}
The test
\begin{equation}\label{test_ratio_P3}
\widetilde{\phi}_n=\mathbbm{1}_{\left\lbrace 2\log\left(\frac{\mathcal{L}\left(\widehat{\theta}_n,X^{(n)}\right)}{\mathcal{L}\left(\theta,X^{(n)}\right)}\right)\geqslant C_{\alpha} \right\rbrace}
\end{equation}
is AUMPI($\alpha$) to test $\theta = \theta_0$ against $\theta \neq \theta_0$ where $C_{\alpha}$ is the $\alpha$-quantile of $\chi_p^2.$
\end{theo}

\begin{proof}
See Section \ref{Preuve_6}
\end{proof}

\section{Conclusion.}

We have seen through this study that the classical properties on the stable autoregressive processes concerning the MLE are preserved despite the harmful effects of the filter which leads to the lack of ergodicity.
On the other hand, the results obtained in \cite{BCK14} are sufficient to deduce the LAN property which leads to the asymptotic efficiency for the estimation of the autoregressive parameter and to the construction of an asymptotically optimal procedure to test the significance of the same parameter.
It would be interesting in the future to extend this study without any assumption on convergence rate of the PACF of the nuisance process. It would be just as interesting to build a procedure to detect a change in the autoregressive dynamic. They would be needed for that convergence rate in (\ref{controle_zeta1_gamma}) and (\ref{beta_zeta_zero}) in order to apply the method proposed in \cite{davis1995}.

\section*{Acknowledgments.}

I would like to thank Alexandre Brouste and Fr\'{e}d\'{e}ric Pro\"{i}a for proposing this very interesting subject and for advising me during this study.

\section{Auxiliary results.}\label{Sec4}

\

This section is devoted to the numerous technicals lemmas and proposition that we will use for the proof of the results of section \ref{Sec2}. Before starting the proofs, we give technical results established in \cite{BCK14}. We can write

\begin{equation}\label{decompo_MLE_P1}
\widehat{\theta}_n-\theta=(\langle M \rangle_n)^{-1}M_n,
\end{equation}

\noindent where

\begin{align*}
\langle M \rangle_n = \sum_{i=1}^n \frac{a_{i-1}^*\zeta_{i-1}\zeta_{i-1}^*a_{i-1}}{\sigma_i^2} \ \text{and} \ M_n=\sum_{i=1}^n \frac{a_{i-1}^*\zeta_{i-1}\varepsilon_i}{\sigma_i}.
\end{align*}

Let $\mathcal{F}_n=\mathcal{F}(X_0,X_1,\dots,X_n)$ be  the $\sigma$-algebra generated by the values of the process up to time $n.$ It follows that $(M_n)$ is a $\mathcal{F}_n$-martingale, and that  $(\langle M \rangle_n)$ is its bracket process.

\

\noindent We also have the following properties :

\begin{equation}\label{beta_sommable_P1}
\sum_{i=1}^{\infty} \beta_i^2 < \infty,
\end{equation}

\begin{equation}\label{proba_inf_Fisher_P1}
{\frac{\langle M \rangle_n}{n}}\xrightarrow[n\to\infty]{\mathbb{P}_{\theta}^{(n)}}{I(\theta)},
\end{equation}
where $I(\theta)$ is defined in (\ref{Lyapu_P1}),

\begin{equation}\label{var_asympto_innov_P1}
{\sigma_n^2}\xrightarrow[n\to\infty]{}{\sigma_{\infty}^2 > 0},
\end{equation}

\begin{equation}\label{loi_martingale_P1}
{\frac{M_n}{\sqrt{n}}}\xrightarrow[n\to\infty]{\mathcal{L}}{\mathcal{N}(0,I(\theta))}
\end{equation}

\noindent under $\mathbb{P}_{\theta}^{(n)}.$

\

\begin{rem}\label{Rem1}
The process $(\zeta_n)$ obtained via (\ref{Process_Filtré_P1}) is a Markov process, but unfortunately inhomogeneous.
The first step in our proofs will be to show that the firt $p$ components of $(\zeta_n)$ have the same asymptotic behavior as an autoregressive process strictly stationary and ergodic.
\end{rem}

\begin{prop}\label{trajectoire_zeta_P2}
Let $(\zeta_n^{(1)})$ be the firt $p$ components of $(\zeta_n)$ defined in (\ref{nouveau_zeta_P1}), and $(\zeta_n^{(2)})$ be the last $p$ components. Consider the process
\begin{align*}
\gamma_n = A_0\gamma_{n-1} + \ell^1\sigma_{\infty}\varepsilon_n, \ \forall n \geqslant 1
\end{align*}
with $\gamma_0$ having the strictly stationnary and ergodic distribution associated with the autoregressive relation, and $\ell^1 = (1,0 \dots,0)^*$ a vector of lenght $p.$ Then,
\begin{equation}\label{controle_zeta1_gamma}
{\left\Vert \zeta_n^{(1)} - \gamma_n \right\Vert}\xrightarrow[n\to\infty]{a.s.}0.
\end{equation}
\end{prop}

\begin{proof}
See Section \ref{Preuve_1}.
\end{proof}

\begin{rem}\label{Rem2}
Since $\rho(A_0)<1$ the process $(\gamma_n)$ admits a unique representation with the following properties : causality, stationarity and ergodicity (see \cite{BD1991} for more details). \\
This property about $(\gamma_n)$ will allow us later via ergodicity arguments to obtain the a.s convergences. Since $(\gamma_n)$ is a Gaussian ergodic process, $\mathbb{E}(\Vert \gamma_n \Vert^p) = K_p < \infty$ for all $p.$ 
\end{rem}

\begin{lem}\label{Lem_Tech1_P4}
Consider a random vector $T_n \in \mathbb{R}^d$ such that, for all $n \geqslant 1,$ 
\begin{align*}
T_n \sim \mathcal{N}(0,A_n)
\end{align*}
where the covariance matrix satisfies $\Vert A_n \Vert = O(n^{-\delta})$ for some $\delta >0.$ Then,
\begin{equation}\label{EQ_4.1}
{\left\Vert T_n \right\Vert}\xrightarrow[n\to\infty]{a.s.}0.
\end{equation}
\end{lem}

\begin{proof}
Following the idea of Lemma A.1 in \cite{CGL13}, let $\varepsilon > 0$ such that
\begin{align*}
\mathbb{P}(\Vert T_n \Vert > \varepsilon) & =  \mathbb{P}(\Vert T_n \Vert^2 > \varepsilon^2) \\
& =  \mathbb{P}(\langle A_n \mu_n, \mu_n \rangle > \varepsilon^2) \leqslant \mathbb{P}\left(\Vert \mu_n \Vert^2 > \frac{\varepsilon^2}{\Vert A_n \Vert}\right),
\end{align*}
where $\Vert \mu_n \Vert^2 \sim \chi^2(d)$ which, in turn, implies
\begin{align*}
\mathbb{P}(\Vert \mu_n \Vert^2 > \varepsilon^2\Vert A_n \Vert^{-1}) = \int_{\varepsilon^2\Vert A_n \Vert^{-1}}^{\infty} c(d)x^{\frac{d}{2}-1}\exp\left(-\frac{x}{2}\right) \, \mathrm{d}x \leqslant \int_{\varepsilon^2\Vert A_n \Vert^{-1}}^{\infty} c(d)x^{-\beta} \, \mathrm{d}x 
\end{align*}
for any $\beta > 0$ and $n$ big enough, where $c(d)$ is some positive constant independent of $x$ and $n$. Making use of the hypothesis on $\Vert A_n \Vert$, we get
\begin{align*}
\int_{\varepsilon^2\Vert A_n \Vert^{-1}}^{\infty} c(d)x^{-\beta} \, \mathrm{d}x = \frac{c(d)\varepsilon^{2(1-\beta)}\Vert A_n\Vert^{\beta-1}}{\beta-1}=O(n^{\delta(1-\beta)})
\end{align*}
as soon as $\beta >1.$ Let  us choose $\beta$ such that $\delta(1-\beta) < -1$, i.e. $\beta > \frac{1+\delta}{\delta}.$
It remains to apply Borel-Cantelli's lemma to reach the desired result.
\end{proof}

\noindent The following lemma will allow us to control the norm of the matrix $\widetilde{A}_n.$ 

\begin{lem}\label{Lem_Tech2_P4}
Let $\widetilde{A}_n$ be the transition matrix as it is defined in (\ref{nouveau_zeta_P1}), namely

$$\forall n \geqslant 1, \  \widetilde{A}_n = \begin{pmatrix}
A_0 & \beta_n A_0 \\
\beta_n Id_p & Id_p
\end{pmatrix}  \ \text{with} \ \widetilde{A}_0 = \begin{pmatrix}
A_0 & A_0 \\
Id_p & Id_p
\end{pmatrix}.$$

\

Then,

\begin{equation}\label{Tech2.1_P4}
\sup\limits_{n \in \mathbb{N}} \left\Vert \prod_{i=1}^n \widetilde{A}_{n-i} \right\Vert  < \infty.
\end{equation}
\end{lem}

\begin{proof}
Let $T_0 = Id_{2p}$ and, for $n \geqslant 1$,
\begin{align*}
T_n = \prod_{i=1}^n \widetilde{A}_{n-i}.
\end{align*}
Working block by block, it is not hard to see that

\begin{align*}
\left\{
    \begin{array}{ll}
       T_{n+1}^{(11)} = A_0T_n^{(11)}+\beta_n A_0T_n^{(21)} \\
       T_{n+1}^{(12)} = A_0T_n^{(12)} + \beta_n A_0T_n^{(22)} \\
       T_{n+1}^{(21)} = \beta_n T_n^{(11)}+T_n^{(21)} \\
       T_{n+1}^{(22)} = \beta_n T_n^{(12)}+T_n^{(21)}
    \end{array}
\right.
\end{align*}
from the recursive equation $T_{n+1}=\widetilde{A}_nT_n,$ where we use the notation
\begin{align*}
T_n = \begin{pmatrix}
T_n^{(11)} & T_n^{(12)} \\
T_n^{(21)} & T_n^{(22)}
\end{pmatrix}.
\end{align*}
Since $\rho(A_0) < 1,$ we know that there exists a matrix norm $\Vert \cdot \Vert_*= \sup(\vert \cdot u \vert_* ; u \in \mathbb{C}^p, \vert u \vert_* = 1)$ satisfying $\Vert A_0 \Vert_* < 1$ (see Proposition 2.3.15 of \cite{Duf97}). It follows that

\begin{align*}
    \left\{
         \begin{array}{ll}
            \Vert T_{n+1}^{(11)} \Vert_* \leqslant a_0\Vert T_n^{(11)}   
             \Vert_*+ \vert \beta_n \vert a_0 \Vert T_n^{(21)} \Vert_* \\
            \Vert T_{n+1}^{(21)} \Vert_* \leqslant \vert \beta_n \vert 
            \Vert T_n^{(11)} \Vert_* + \Vert T_n^{(21)} \Vert_*
         \end{array}
    \right.
\end{align*} 
where, for a better readability, we set $a_0=\Vert A_0 \Vert_*.$ From the previous relations, there is some constant $C >0$ such that
\begin{equation}\label{major_lem4.2}
\Vert T_{n+1}^{(11)} \Vert_* + \Vert T_n^{(21)} \Vert_* \leqslant C\prod_{i=1}^n \Vert H_i \Vert,
\end{equation}
where
\begin{align*}
\forall i \geqslant 1, \ H_i = \begin{pmatrix}
a_0 & \vert \beta_i \vert a_0 \\
\vert \beta_i \vert & 1
\end{pmatrix}.
\end{align*}
Now, to evaluate $\Vert H_i \Vert,$ one is going to study the spectrum of $H_i^*H_i.$ Let 
\begin{align*}
P_i(\lambda)=\lambda^2 - \lambda(1+a_0^2)(1+\beta_i^2) + a_0^2(1-\beta_i^2)^2
\end{align*}
be the characteristic polynomial of $H_i^*H_i$ defined for all $ \lambda \in \mathbb{C}.$ Then, a straightforward calculation gives 
\begin{equation}\label{determinant_lem4.2}
\Delta_i= [(1-a_0)^2+\beta_i^2(1+a_0)^2][(1+a_0)^2 + \beta_i(1-a_0)^2]>0
\end{equation}
as discriminant of the equation $P_i(\lambda)=0,$ thus leading to real eingenvalues 
\begin{align*}
\lambda_{k,i}= \frac{(1-a_0^2)(1+\beta_i^2) \pm \sqrt{\Delta_i}}{2}, \ k=1,2.
\end{align*}
Since we recall that $\beta_n \rightarrow 0,$ a Taylor expansion of $\sqrt{\Delta_i}$ enables to write

\begin{align*}
\lambda_{k,i} = \frac{(1+a_0^2)(1+\beta_i^2)}{2} \pm \frac{1-a_0^2}{2}[1+K_i\beta_i^2+o(\beta_i^2)]
\end{align*}
for some easily identifiable $\vert K_i \vert =O(1),$ as $i \rightarrow \infty.$ Consequently,
\begin{align*}
\lambda_{1,i}=a_0^2+O(\beta_i^2) \ \text{and} \ \lambda_{2,i} = 1 + O(\beta_i^2)
\end{align*}
which clearly gives $\Vert H_i \Vert^2 = 1 + O(\beta_i^2).$ This rate together with (\ref{major_lem4.2}) and (\ref{beta_sommable_P1}) are sufficient to get $\Vert T_{n+1}^{(11)} \Vert_* + \Vert T_{n+1}^{(21)} \Vert_* < \infty.$ Since the same reasoning holds for $\Vert T_{n+1}^{(12)} \Vert_* + \Vert T_{n+1}^{(22)} \Vert_*,$ the desired result is proved through the definition of $T_n$.
\end{proof}

\noindent The next lemma is interested in the a.s. convergence of $\left(\beta_n\zeta_n\right)$.

\begin{lem}\label{lem_Tech3_P4}
Consider the process $(\zeta_n)$ defined in (\ref{nouveau_zeta_P1}). Then,
\begin{equation}\label{beta_zeta_zero}
{\Vert \beta_n \zeta_n \Vert}\xrightarrow[n\to\infty]{a.s.}0.
\end{equation}
\end{lem}

\begin{proof}
The autoregressive relation leads to
\begin{align*}
\zeta_n = \sum_{k=1}^{n-1} \left(\prod_{i=1}^k \widetilde{A}_{n-i}\right)\ell \sigma_{n-k}\varepsilon_{n-k} + \ell \sigma_n\varepsilon_n.
\end{align*}
Thus,
\begin{align*}
\mathbb{E}\left(\beta_n^2\zeta_n\zeta_n^*\right) = & \beta_n^2 \sum_{k=1}^n P_{k,n}\ell\sigma_{n-k}^2\mathbb{E}(\varepsilon_{n-k}^2)\ell^*P_{k,n}^* + \beta_n^2\ell \ell^*\sigma_n^2\mathbb{E}(\varepsilon_n^2) \\
= & \beta_n^2 \left( \sum_{k=1}^n P_{k,n}\ell \sigma_{n-k}^2\ell^*P_{k,n}^* + \ell\sigma_n^2\ell^*\right)
\end{align*}
where, for $n-1 \geqslant k \geqslant 1, \ P_{k,n} =\widetilde{A}_{n-1} ... \widetilde{A}_{n-k}.$ From Lemma \ref{Lem_Tech2_P4},
\begin{align*}
\left\Vert \mathbb{E}\left(\beta_n^2\zeta_n\zeta_n^*\right) \right\Vert \leqslant \beta_n^2\left(\sum_{k=1}^n \Vert P_{k,n} \Vert^2 \sigma_{n-k}^2 + \sigma_n^2\right) \leqslant Kn\beta_n^2
\end{align*}
for some $K>0$ and a sufficiently large $n.$ Thus, from hypothesis $(H_3),$ 
\begin{align*}
\left\Vert \mathbb{E}\left(\beta_n^2\zeta_n\zeta_n^*\right) \right\Vert \leqslant \frac{K}{n^{\delta}}
\end{align*}
for some $\delta>0,$ and Lemma \ref{Lem_Tech1_P4} gives the result.
\end{proof}

\noindent We can now use the previous Lemmas to prove the Proposition \ref{trajectoire_zeta_P2}.

\subsection{Proof of Proposition \ref{trajectoire_zeta_P2}}\label{Preuve_1}
\begin{proof}
The direct calculation shows that
\begin{align*}
\zeta_n^{(1)} - \gamma_n = A_0(\zeta_{n-1}^{(1)} - \gamma_{n-1}) + \beta_{n-1}A_0\zeta_{n-1}^{(2)} + \ell^1(\sigma_n - \sigma_{\infty})\varepsilon_n.
\end{align*}
From Lemma \ref{lem_Tech3_P4}, ${\Vert \beta_{n-1}A_0\zeta_{n-1}^{(2)} \Vert}\xrightarrow[n\to\infty]{a.s.}0$ and, from (\ref{var_asympto_innov_P1}) and the normality of  $(\varepsilon_n), \\ {\Vert \ell^1(\sigma_n-\sigma_{\infty})\varepsilon_n \Vert}\xrightarrow[n\to\infty]{a.s.}0.$ Hence, using the same norm as the proof of Lemma \ref{Lem_Tech2_P4}, for all $\eta >0,$ there exists a random $n_0$ such that, for all $n \geqslant n_0,$
\begin{align*}
\left\Vert \zeta_n^{(1)} - \gamma_n \right\Vert_* \leqslant \Vert A_0 \Vert_* \left\Vert \zeta_{n-1}^{(1)} - \gamma_{n-1} \right\Vert_* + \eta \ \ a.s.
\end{align*}
Since $\Vert A_0 \Vert_* <1,$ we conclude that ${\Vert \zeta_n^{(1)} - \gamma_n\Vert}\xrightarrow[n\to\infty]{a.s.}0.$
\end{proof}

\begin{lem}\label{Tech4.1}
Under $\mathbb{P}_{\theta_0}^{(n)}$ we have,
\begin{equation}\label{loi_ratio_tech4.1_P4}
{2\log\left( \frac{\mathcal{L}(\widehat{\theta}_n,X^{(n)})}{\mathcal{L}(\theta_0,X^{(n)})}\right)}\xrightarrow[n\to\infty]{\mathcal{L}}{\chi_p^2}.
\end{equation}
\end{lem}

\begin{proof}
By using the decomposition in the proof of Theorem \ref{Theo_LAN_AR(p)_P3}, we have,
\begin{align*}
2\log\left( \frac{\mathcal{L}(\widehat{\theta}_n,X^{(n)})}{\mathcal{L}(\theta_0,X^{(n)})}\right) = & 2\big\langle \sqrt{n}(\widehat{\theta}_n-\theta_0), \frac{M_n}{\sqrt{n}}\big\rangle - \big\langle \sqrt{n}(\widehat{\theta}_n-\theta_0), I(\theta_0)\sqrt{n}(\widehat{\theta}_n-\theta_0)\big\rangle \\
& - \big\langle \sqrt{n}(\widehat{\theta}_n-\theta_0),(\frac{\langle M \rangle_n}{n} - I(\theta_0))\sqrt{n}(\widehat{\theta}_n-\theta_0)\big\rangle \\
= & \big\langle \langle M \rangle_n^{-\frac{1}{2}} M_n,\langle M \rangle_n^{-\frac{1}{2}} M_n \big\rangle.
\end{align*}
Thus (\ref{loi_ratio_tech4.1_P4}) follows immediately from (\ref{proba_inf_Fisher_P1}) and (\ref{loi_martingale_P1}).
\end{proof}

\noindent The following lemma can be seen as a matrix Toeplitz lemma, in some sense (see Theorem 1.1 in \cite{LH2017}).

\begin{lem}\label{lem_Toep}
Let $(B_{k,n})$ and $(A_n)$ be two sequences of square matrices such that
\begin{itemize}
      \item ${\Vert A_n - A \Vert}\xrightarrow[n\to\infty]{}0,$
      \item ${\Vert \sum_{k=1}^n B_{k,n} - B \Vert}\xrightarrow[n\to\infty]{}0,$
      \item $\sum_{k=1}^n \Vert B_{k,n} \Vert$ is bounded with respect to $n,$
      \item for all $n_0 >0$, \ ${\sum_{k=1}^{n_0} \Vert B_{k,n} \Vert}\xrightarrow[n\to\infty]{}0.$
\end{itemize}
Then,
\begin{equation}\label{CV_Toep}
{ \left\Vert \sum_{k=1}^n A_kB_{k,n} -AB \right\Vert}\xrightarrow[n\to\infty]{}0.
\end{equation}
\end{lem}

\begin{proof}
We have,
\begin{align*}
\sum_{k=1}^n A_kB_{k,n} -AB = \sum_{k=1}^n A_kB_{k,n} - \sum_{k=1}^n AB_{k,n} + \sum_{k=1}^n AB_{k,n} - AB.
\end{align*}
For all $\varepsilon >0,$ one can find $n_0 >0$ such that, for all $n \geqslant n_0, \ \Vert A_n - A\Vert < \varepsilon.$  Thus,
\begin{align*}
\left\Vert \sum_{k=1}^n A_kB_{k,n} -AB \right\Vert \leqslant & \left\Vert \sum_{k=1}^n (A_k-A)B_{k,n} \right\Vert + \left\Vert \sum_{k=1}^n A(B_{k,n}-B) \right\Vert \\
 \leqslant & \sum_{k=1}^{n_0-1} \Vert A_k - A \Vert \Vert B_{k,n} \Vert + \varepsilon \sum_{k=n_0}^n \Vert B_{k,n} \Vert + \Vert A \Vert \left\Vert \sum_{k=1}^n B_{k,n} - B \right\Vert.
\end{align*}
The combination of the whole hypotheses enables to show that the right-hand side of latter expression can be made arbritrarily small, as $n$ tends to infinity.
\end{proof}

\begin{rem}\label{Rem5}
The last result can be easily extended when hypotheses holds $a.s$.
\end{rem}

\begin{rem}\label{Rem6}
Given two sequences of vectors $(u_n)$ and $(v_n) \in \mathbb{R}^d.$ We have $u_n v_n^* =\overline{u}_n\overline{v}_n^*,$ where 

\begin{align*}
\overline{u}_n = \begin{pmatrix}
u_{11} & 0 & \dots & \dots & 0 \\
u_{21} & 0 & \dots & \dots & 0 \\
\dots & \dots & \dots & \dots & \dots \\
u_{(d-1)1} & 0 & \dots & \dots & 0 \\
u_{d1} & 0 & \dots & \dots & 0 
\end{pmatrix}.
\end{align*}

\noindent More precisely, $\overline{u}_n$ is $d \times d$ matrix with $u_n$ for first column and 0 elsewhere. Then $\Vert u_n v_n^* \Vert \leqslant \Vert \overline{u}_n \Vert \Vert \overline{v}_n^* \Vert,$ and, since we work in finite dimension, ${\Vert \overline{u}_n - \overline{u}\Vert}\xrightarrow[n\to\infty]{}0$ if and only if each component of $\overline{u}_n - u$ converges to $0.$
\end{rem}

\begin{lem}\label{A0_ouvert}
$\Theta$ is an open subset of $\mathbb{R}^p.$
\end{lem}

\begin{proof}
Let
\begin{align*}
P(z)=1-\theta_1z - \dots \theta_p z^p,
\end{align*}
a polynomial function defined on $\mathbb{C}.$ $(H_1)$ is equivalent at the following condition (see \cite{BD1991}) : 
\begin{equation}\label{reformu_parametric}
\text{if} \ \vert z \vert \leqslant 1 \ \text{then} \ \vert P(z) \vert > 0.
\end{equation}
\noindent Let 
\begin{align*}
\alpha=(\alpha_1,\alpha_2,\dots,\alpha_p) \in \mathbb{R}^p \ \text{and} \ Q_{\alpha}(z)=1-(\theta_1+\alpha_1)z - \dots - (\theta_p + \alpha_p)z^p,
\end{align*}
then, for all $\vert z \vert \leqslant 1,$
\begin{align*}
\vert Q_{\alpha}(z)\vert & = \vert P(z) - \alpha_1 z - \dots - \alpha_p z^p \vert \\
& \geqslant \vert P(z) \vert - \vert \alpha_1 z + \dots + \alpha_p z^p \vert \\
& \geqslant \vert P(z) \vert - \vert \alpha_1 \vert - \dots - \vert \alpha_p \vert.
\end{align*}
Condition (\ref{reformu_parametric}) ensures that
\begin{align*}
\inf\limits_{z \in D(0;1)} \ \vert P(z) \vert = \delta_{min} > 0.
\end{align*}
It remains to choose $\vert \alpha_1 \vert + \dots + \vert \alpha_p \vert < \delta_{min}$ in order to reach the desired result. 
\end{proof}

\begin{rem}\label{Stabilité_espace_param}
The last Lemma ensures that if $\theta \in \Theta$ then for all $u \in \mathbb{R}^p$ and for $n$ big enough, $\theta + \frac{u}{\sqrt{n}} \in \Theta.$
Take the notation of the last proof and choose $\alpha \in \mathbb{R}^p$ such that $\Vert \alpha \Vert < \delta_{min},$ then, for any rotation $R$ from $\mathbb{R}^p$ to $\mathbb{R}^p$ and $n$ big enough, $\theta + \frac{Ru}{\sqrt{n}} \in \Theta.$
\end{rem}

\section{Proofs of the mains results.}\label{Sec5}♣

\

\subsection{Proof of Theorem \ref{Theo1_P2}}\label{Preuve_2}
\begin{proof}

We have,
\begin{align*}
\frac{\langle M \rangle_n}{n} \ = & \ \frac{1}{n} \sum_{i=1}^n \frac{(\zeta_{i-1}^{(1)}+\beta_{i-1}\zeta_{i-1}^{(2)}-\gamma_{i-1}+\gamma_{i-1})(\zeta_{i-1}^{(1)}+\beta_{i-1}\zeta_{i-1}^{(2)}-\gamma_{i-1}+\gamma_{i-1})^*}{\sigma_i^2} \\
= & \ S_n + r_n
\end{align*}

\noindent where

\begin{align*}
S_n = \frac{1}{n} \sum_{k=1}^n \frac{\gamma_{i-1}\gamma_{i-1}^*}{\sigma_i^2}.
\end{align*}
The remainder term $r_n$ is shown to be negligible via Ces\`{a}ro's theorem as well as the ergodicty of $(\gamma_n)$, Lemma \ref{lem_Tech3_P4}, Proposition \ref{trajectoire_zeta_P2}, Lemma \ref{lem_Toep} and Remark \ref{Rem6}. A direct application of the ergodic theorem together with Lemma \ref{lem_Toep} leads to
\begin{align*}
{S_n}\xrightarrow[n\to\infty]{a.s.}{\frac{1}{\sigma_{\infty}^2}\mathbb{E}(\gamma_0\gamma_0^*)},
\end{align*}
and finally thanks to (\ref{proba_inf_Fisher_P1}),
\begin{align*}
{S_n}\xrightarrow[n\to\infty]{a.s.}{I(\theta)}.
\end{align*}
Thus,
\begin{equation}\label{conv_crochet_ps}
{\frac{\langle M \rangle_n}{n}}\xrightarrow[n\to\infty]{a.s.}{I(\theta)}.
\end{equation}
By using the fact that $\varepsilon_n$ is independent of $\gamma_{n-1}$ and similar arguments as in this proof, we have,
\begin{equation}\label{TH1.4_P4}
{\frac{M_n}{n}}\xrightarrow[n\to\infty]{a.s.}0,
\end{equation}
leading to the strong consistency

\begin{align*}
{\widehat{\theta}_n}\xrightarrow[n\to\infty]{a.s.}{\theta}.
\end{align*}
\end{proof}

\noindent From now on, let $\bar{\mathcal{F}}_n$ be the $\sigma$-algebra $\bar{\mathcal{F}}_n=\bar{\mathcal{F}}_n(X_0, \dots, X_n, \gamma_0, \dots, \gamma_n)$ where, $(\gamma_n)$ is the process defined in Proposition \ref{trajectoire_zeta_P2}.

\begin{rem}\label{rem_new_filtration}
$(M_n)$ is a $\bar{\mathcal{F}}_n$-martingale and the introduction of $\bar{\mathcal{F}}_n$ is necessary in the following proof.
\end{rem}

\subsection{Proof of Theorem \ref{Theo2_LFQ_P2}}\label{Preuve_3}
\begin{proof}
In order to prove Theorem \ref{Theo2_LFQ_P2}, we will use the quadratic stong law for martingales (see Theorem 2.1 in \cite{CM00}).
Take $V_n = \sqrt{n}Id_p,$ a sequence of regular matrices in the sense of Chaabane and Maouia. Now, we studing the asymptotic behavior of
\begin{equation}\label{TH2.1_P4}
\frac{[M]_n}{n}= \frac{1}{n} \sum_{i=1}^n \left(\frac{(\zeta_{i-1}^{(1)}+\beta_{i-1}\zeta_{i-1}^{(2)}-\gamma_{i-1}+\gamma_{i-1})(\zeta_{i-1}^{(1)}+\beta_{i-1}\zeta_{i-1}^{(2)}-\gamma_{i-1}+\gamma_{i-1})^*}{\sigma_i^2}\right)\varepsilon_i^2.
\end{equation}
By using similar arguments as in  the previous proof we have,
\begin{equation}\label{TH2.2_P4}
{\frac{[M]_n}{n}}\xrightarrow[n\to\infty]{a.s.}{\mathbb{E}\left(\frac{\gamma_1\gamma_1^*}{\sigma_{\infty}^2}\right)}.
\end{equation}

\noindent Let us now look at the Lindeberg's condition, we have to show that, for all $\varepsilon>0,$

\begin{equation}\label{Lindeberg_condition}
{L_n=\frac{1}{n} \sum_{i=1}^n \mathbb{E}\left(\Vert \Delta M_i \Vert^2 \mathbbm{1}_{\left\lbrace \Vert \Delta M_i \Vert \geqslant \varepsilon\sqrt{n}\right\rbrace} \middle\vert \bar{\mathcal{F}}_{i-1}\right)}\xrightarrow[n\to\infty]{a.s.}0.
\end{equation}
\noindent Let $M>0$ and

\begin{align*}
L_{n,M} = \frac{1}{n} \sum_{i=1}^n \mathbb{E}\left(\Vert \Delta M_i \Vert^2 \mathbbm{1}_{\left\lbrace \Vert \Delta M_i \Vert \geqslant M\right\rbrace} \middle\vert \bar{\mathcal{F}}_{i-1}\right).
\end{align*}

\noindent From (\ref{decompo_MLE_P1}) and the definition of $a_n,$ we have $\Delta M_1=M_1$ and, for $n \geqslant 2,$

\begin{align*}
\Delta M_n = \frac{(\zeta_{n-1}^{(1)} + \beta_{n-1}\zeta_{n-1}^{(2)})\varepsilon_n}{\sigma_n}.
\end{align*}

\noindent It follows that

\begin{align*}
\Vert \Delta M_n \Vert^2 = & \left\Vert \frac{(\zeta_{n-1}^{(1)}-\gamma_{n-1} + \beta_{n-1}\zeta_{n-1}^{(2)}+\gamma_{n-1})\varepsilon_n}{\sigma_n}\right\Vert^2 \\
\leqslant & \frac{2\left\Vert\zeta_{n-1}^{(1)}-\gamma_{n-1} + \beta_{n-1}\zeta_{n-1}^{(2)}\right\Vert^2 \varepsilon_n^2}{\sigma_n^2}+\frac{2\Vert \gamma_{n-1}\Vert^2\varepsilon_n^2}{\sigma_n^2}.
\end{align*}

\noindent Then,

\begin{align*}
L_{n,M} \leqslant 2R_{1,n} + 2R_{2,n} + 2R_{3,n},
\end{align*}

\noindent where,

\begin{align*}
R_{1,n} = \frac{1}{n} \sum_{i=1}^n \frac{\left\Vert\zeta_{i-1}^{(1)}-\gamma_{i-1} + \beta_{i-1}\zeta_{i-1}^{(2)}\right\Vert^2}{\sigma_i^2},
\end{align*}

\begin{align*}
R_{2,n} = \frac{1}{n} \sum_{i=1}^n \mathbb{E}\left( \frac{\Vert \gamma_{i-1}\Vert^2}{\sigma_i^2} \varepsilon_i^2 \mathbbm{1}_{\left\lbrace \frac{\Vert \gamma_{i-1} \Vert \vert \varepsilon_i \vert}{\sigma_i} \geqslant \frac{M}{2} \right\rbrace} \middle\vert \bar{\mathcal{F}}_{i-1}\right),
\end{align*}

\noindent and,

\begin{align*}
R_{3,n} = \frac{1}{n} \sum_{i=1}^n \mathbb{E}\left( \frac{\Vert \gamma_{i-1} \Vert^2}{\sigma_i^2} \varepsilon_i^2 \mathbbm{1}_{\left\lbrace \frac{\left\Vert\zeta_{i-1}^{(1)}-\gamma_{i-1} + \beta_{i-1}\zeta_{i-1}^{(2)}\right\Vert \vert \varepsilon_i \vert}{\sigma_i} \geqslant \frac{M}{2} \right\rbrace} \middle\vert \bar{\mathcal{F}}_{i-1} \right).
\end{align*}

\noindent The same reasoning as above shows that $R_{1,n}$ tend to $0,$ $a.s.$ Let us focus on the more intricate terms $R_{2,n}$ and $R_{3,n}.$ First, we know from (\ref{var_asympto_innov_P1}) that, for some $0 < \widetilde{m} < \sigma_{\infty},$ there exists $n_0$ such that, for $n \geqslant n_0,$
 $\vert \sigma_n \vert \geqslant \widetilde{m}.$ Hence,

\begin{align*}
R_{2,n} \leqslant R_{2,n}^{(0)}+\frac{1}{n}\sum_{i=n_0}^n \mathbb{E}\left( \frac{\Vert \gamma_{i-1}\Vert^2}{\widetilde{m}^2} \varepsilon_i^2 \mathbbm{1}_{\left\lbrace \frac{\Vert \gamma_{i-1} \Vert \vert \varepsilon_i \vert}{\widetilde{m}} \geqslant \frac{M}{2} \right\rbrace} \middle\vert \bar{\mathcal{F}}_{i-1}\right),
\end{align*}
where obviously, ${R_{2,n}^{(0)}}\xrightarrow[n\to\infty]{a.s.}0.$ The process $(\gamma_{n-1}\varepsilon_n)$ is ergodic since  $(\varepsilon_n)$ is $i.i.d.$ and since there exist $\phi$ independent of $n$ such that $\gamma_{n-1} = \phi(\varepsilon_{n-1},\varepsilon_{n-2}, \dots)$ \ (See theorem 5.3.8 in \cite{stout74}). Thus,

\begin{align*}
\lim \limits_{n\to\infty} R_{2,n} \leqslant \mathbb{E}\left( \frac{\Vert \gamma_0 \Vert^2}{\widetilde{m}^2}\varepsilon_1^2 \mathbbm{1}_{\left\lbrace \frac{\Vert \gamma_0 \Vert}{\widetilde{m}} \vert \varepsilon_1\vert \geqslant \frac{M}{2} \right\rbrace} \middle\vert \bar{\mathcal{F}}_{i-1} \right),
\end{align*}
and letting $M \rightarrow \infty$ leads to ${R_{2,n}}\xrightarrow[n,M\to\infty]{a.s.}0.$ Then, by Cauchy-Schwarz's inequality,

\begin{align*}
R_{3,n}^2 \leqslant & \frac{1}{n^2} \sum_{i=1}^n \mathbb{E}\left( \frac{\Vert \gamma_{i-1} \Vert^4}{\sigma_i^4} \varepsilon_i^4 \middle\vert \bar{\mathcal{F}}_{i-1} \right)\sum_{i=1}^n \mathbb{E}\left(\mathbbm{1}_{\left\lbrace \frac{\left\Vert \zeta_{i-1}^{(1)} - \gamma_{i-1} + \beta_{i-1}\zeta_{i-1}^{(2)}\right\Vert \vert \varepsilon_i \vert}{\vert \sigma_i \vert}\geqslant \frac{M}{2} \right\rbrace} \middle\vert \bar{\mathcal{F}}_{i-1} \right)\\
\leqslant & \frac{3}{n^2} \sum_{i=1}^n \frac{\Vert \gamma_{i-1}\Vert^4}{\sigma_i^4} \sum_{i=1}^n \left[ \mathbbm{1}_{\left\lbrace \frac{\left\Vert \zeta_{i-1}^{(1)} - \gamma_{i-1} + \beta_{i-1}\zeta_{i-1}^{(2)}\right\Vert}{\vert\sigma_i\vert} \geqslant \frac{\sqrt{M}}{\sqrt{2}}\right\rbrace}+\mathbb{P}\left(\vert\varepsilon_1\vert \geqslant \frac{\sqrt{M}}{\sqrt{2}}\right)\right]. 
\end{align*}

\noindent Since ${\Vert \zeta_{i-1}^{(1)} - \gamma_{i-1} + \beta_{i-1}\zeta_{i-1}^{(2)} \Vert}\xrightarrow[n\to\infty]{a.s.}0$ by Proposition \ref{trajectoire_zeta_P2} and Lemma \ref{lem_Tech3_P4}, and since $(\varepsilon_n)$ is Gaussian, another application of the ergodic theorem is sufficient to ensure that ${R_{3,n}}\xrightarrow[n\to\infty]{a.s.}0,$ letting again $M \rightarrow \infty.$ Thus,

\begin{align*}
{L_{n,M}}\xrightarrow[n,M\to\infty]{a.s.}0 \ \text{and so} \ {L_{n}}\xrightarrow[n\to\infty]{a.s.}0.
\end{align*}
By the quadratic strong law for martingales,
\begin{equation}\label{TH2.5_P4}
{\frac{1}{p \log n}\sum_{k=1}^n \left(1- \frac{k^p}{(k+1)^p}\right) \frac{M_kM_k^*}{k}}\xrightarrow[n\to\infty]{a.s.}{I(\theta)}.
\end{equation}

\noindent Let $\lambda^{(k)}=\rho(M_kM_k^*)=\Vert M_kM_k^* \Vert \leqslant Tr(M_kM_k^*)$ because the matrices $M_kM_k^*$ are positives. \\
Put $v_i=(0,\dots,0,1,0,\dots,0)^*$ a vector of length $p$ with $1$ at the $i$-th coordinate and $0$ elsewhere. Then,

\begin{equation}\label{TH2.5_P5}
\frac{1}{p \log n}\sum_{k=1}^n \left(1- \frac{k^p}{(k+1)^p}\right)\frac{\lambda^{(k)}}{k} \leqslant \frac{1}{p \log n}\sum_{i=1}^{p} \sum_{k=1}^n \left(1- \frac{k^p}{(k+1)^p}\right)\frac{\langle v_i, M_kM_k^* v_i \rangle}{k} \ a.s.
\end{equation}

\noindent Since the right member of (\ref{TH2.5_P5}) is bounded, the conditions of Lemma \ref{lem_Toep} are satisfied, and we have,

\begin{equation}\label{TH2.6_P4}
\frac{1}{\log n}\sum_{k=1}^n (\widehat{\theta}_k-\theta)(\widehat{\theta}_k-\theta)^*= \frac{1}{\log n}\sum_{k=1}^n \langle M \rangle_k^{-1}M_kM_k^*\langle M \rangle_k^{-1}.
\end{equation}
Since $\left(1- \frac{k^p}{(k+1)^p}\right) \sim \frac{p}{k}$ by using (\ref{TH2.5_P4}), (\ref{TH2.6_P4}), (\ref{conv_crochet_ps}) together with Lemma \ref{lem_Toep}, we obtain,
\begin{equation}\label{TH2.7_P4}
{\frac{1}{\log n} \sum_{k=1}^n (\widehat{\theta}_k-\theta)(\widehat{\theta}_k-\theta)^*}\xrightarrow[n\to\infty]{a.s.}{I(\theta)^{-1}}.
\end{equation}
\end{proof}

\subsection{Proof of Proposition \ref{Prop_LLI}}\label{Preuve_3.1}
\begin{proof}
To prove the law of iterated logarithm, one is going to apply Lemma C.2 of \cite{Ber98}. We have already etablished (\ref{conv_crochet_ps}), so it remains to show that
\begin{align*}
\sum_{n=1}^{\infty} \left( \frac{\left\Vert a_n^* \zeta_n \right\Vert}{\vert \sigma_{n+1} \vert \sqrt{n}}\right)^{\beta} < \infty \ a.s.
\end{align*}
for some $\beta>2.$ From Proposition \ref{trajectoire_zeta_P2} and Lemma \ref{lem_Tech3_P4}, there exists a sequence $( \tau_n)$ such that
\begin{align*}
\Vert a_n^* \zeta_n \Vert = \Vert \gamma_n + \tau_n \Vert \ \text{and} \ \tau_n \rightarrow 0 \ a.s.
\end{align*}
Consequently, for a sufficiently large $n,$
\begin{align*}
\left( \frac{\left\Vert a_n^*\zeta_n\right\Vert}{\vert \sigma_{n+1}\vert \sqrt{n}}\right)^{\beta} \leqslant \frac{\Vert \tau_n + \gamma_n \Vert^{\beta}}{\widetilde{m}^{\frac{\beta}{2}}n^{\frac{\beta}{2}}} \leqslant \frac{\left(1+\Vert \gamma_n \Vert \right)^{\beta}}{\widetilde{m}^{\frac{\beta}{2}}n^{\frac{\beta}{2}}} \ a.s.
\end{align*}
where $\widetilde{m}$ is chosen as in the preceding proof. Since $(\gamma_n)$ is a Gaussian ergodic process, $\Vert \gamma_n \Vert = O(n^{\alpha})$ a.s for all $0 < \alpha < \frac{1}{2}$. The last inequality leads for $n$ sufficiently large to

\begin{align*}
\left( \frac{\left\Vert a_n^*\zeta_n\right\Vert}{\vert \sigma_{n+1}\vert \sqrt{n}}\right)^{\beta} \leqslant \frac{C}{\widetilde{m}^{\frac{\beta}{2}}n^{\beta\frac{1-2\alpha}{2}}} \ a.s.
\end{align*}
for a constant $C>0.$ The desired result follows when $\beta > \frac{2}{1-2\alpha}>2.$ Hence,
\begin{align*}
\limsup\limits_{n \to \infty} \left( \frac{n}{2\log \log n} \right)^{\frac{1}{2}} v^*{\langle M \rangle_n}^{-1}M_n & =  - \liminf\limits_{n \to \infty} \left( \frac{n}{2\log \log n} \right)^{\frac{1}{2}} v^*{\langle M \rangle_n}^{-1}M_n \\
& = (v^*I(\theta)^{-1}v)^{\frac{1}{2}} \ a.s.
\end{align*}
Let ${\widehat{\theta}_n}^i$ be the $i$-th component of $\widehat{\theta}_n$ and $v_i=(0,\dots,0,1,0,\dots,0)^*$ a vector of length $p$ with $1$ at the $i$-th coordinate and $0$ elsewhere. The inequality above implies immediatly (\ref{vitesse_MLE}). 
\end{proof}

\subsection{Proof of Theorem \ref{Theo_LAN_AR(p)_P3}}\label{Preuve_4}
\begin{proof}
Let 

\begin{align*}
\phi_n (\theta_0)=\frac{Id_p}{\sqrt{n}}, u=(u_1,u_2,\dots,u_p)^*,
\end{align*}

\begin{align*}
U_n=\begin{pmatrix}
u_1 & u_2 & \dots & u_p & \beta_n u_1 & \beta_n u_2 & \dots & \beta_n u_p \\
0 & 0 & \dots & \dots & \dots & \dots & \dots & 0 \\
\dots & \dots & \dots & \dots & \dots & \dots & \dots \\
0 & 0 & \dots & \dots & \dots & \dots & \dots & 0
\end{pmatrix} \ \text{and} \ \ell^*U_n= u^*a_n^*.
\end{align*}

\noindent Then,

\begin{align*}
\log \left( \frac{\mathcal{L}(\theta_0+\phi_n (\theta_0)u,X^{(n)})}{\mathcal{L}(\theta_0,X^{(n)})}\right) = & -\frac{1}{2}\sum_{k=1}^n \frac{\ell^*(2\zeta_k-2\widetilde{A}_{k-1}\zeta_{k-1}-\frac{U_{k-1}}{\sqrt{n}}\zeta_{k-1})\ell^*(-\frac{U_{k-1}}{\sqrt{n}}\zeta_{k-1})}{\sigma_k^2} \\
= & -\frac{1}{2} \sum_{k=1}^n \frac{\ell^*(2\ell\sigma_k\varepsilon_k - \frac{U_{k-1}}{\sqrt{n}}\zeta_{k-1})\ell^*(- \frac{U_{k-1}}{\sqrt{n}}\zeta_{k-1})}{\sigma_k^2} \\
= & \sum_{k=1}^n \left(\frac{u^*a_{k-1}^*\zeta_{k-1}\varepsilon_k}{\sqrt{n}\sigma_k} - \frac{1}{2}\frac{u^*a_{k-1}^*\zeta_{k-1}\zeta_{k-1}^*a_{k-1}u}{n\sigma_k^2}\right) \\
= & \big\langle u, \frac{M_n}{\sqrt{n}} \big\rangle - \frac{1}{2}\big\langle u, \frac{\langle M \rangle_n}{n} u \big\rangle \\
= & \big\langle u, \frac{M_n}{\sqrt{n}} \big\rangle - \frac{1}{2}\big\langle u, I(\theta_0)u \big\rangle - \frac{1}{2} \big\langle u,(\frac{\langle M \rangle_n}{n}-I(\theta_0)),u \big\rangle.
\end{align*}
Let $R_n(\theta_0,n) = \frac{1}{2}\big\langle u, (I(\theta_n) - \frac{\langle M \rangle_n}{n})u \big\rangle,$ then $|R_n(\theta_0,n)| \leqslant \frac{1}{2}\Vert I(\theta_0) - \frac{\langle M \rangle_n}{n} \Vert R^2$ and the LAN property in $\theta_0$ follows immediatly from (\ref{loi_martingale_P1}) and (\ref{proba_inf_Fisher_P1}).
\end{proof}

\begin{rem}\label{Rem3} In \cite{IK81}, the LAN property is defined for stastistical experiments admitting decomposition (\ref{LAN_repre_P3}) with $J(\theta_0)=Id_d$. Is not a real restriction since the decomposition (\ref{LAN_repre_P3}) can be reformulated as the Definition 2.1, chapter 2 in \cite{IK81} by letting $v = J(\theta_0)^{-\frac{1}{2}}u$ or equivalently considering a new sequence of matrix rate given by $\widetilde{\phi}_n(\theta_0) = \phi_n(\theta_0)J(\theta_0)^{-\frac{1}{2}}.$
\end{rem}

\subsection{Proof of Proposition \ref{eff_MLE_P3}}\label{Preuve_5}
\begin{proof}
The result follows from (\ref{loi_est1_P1}). More precisely the following condition to obtained the equality in Lemma 13.1, chapter 2 of \cite{IK81}
\begin{equation}\label{Effi_MLE}
{I(\theta_0)^{\frac{1}{2}} \sqrt{n}(\widehat{\theta}_n-\theta_0) - I(\theta_0)^{-\frac{1}{2}} \frac{M_n}{\sqrt{n}}}\xrightarrow[n\to\infty]{\mathbb{P}_{\theta_0}^{(n)}}0
\end{equation}
is satisfied since the MLE is asymptotically efficient in Fisher's sense. 
\end{proof}

\begin{defi}\label{test_equivalent}
A test $\phi_n^1$ is asymptotically equivalent to $\widetilde{\phi}_n^1$ if
\begin{equation}\label{test_equivalent_P3}
{\phi_n^1-\widetilde{\phi}_n^1}\xrightarrow[n\to\infty]{\mathbb{P}_{\theta_0}^{(n)}}0.
\end{equation}
\end{defi}

\noindent We can now state the theorem characterizing the tests AUMPI($\alpha$) established in Theorem 3 of \cite{CHA96}.

\begin{theo}\label{Theo_AUMPI_test_P3}
Take notation of definition \ref{LAN_property_P3} and denote by $C_{\alpha}$ the $\alpha$-quantile of $\chi_d^2.$ Then, the test
\begin{equation}\label{test_AUMPI_P3}
\phi_n=\mathbbm{1}_{ \left\lbrace \langle J(\theta_0)^{-1}Z_n(\theta_0),Z_n(\theta_0)\rangle \geqslant C_{\alpha} \right\rbrace }
\end{equation} 
is AUMPI($\alpha$) to test $\theta=\theta_0$ against $\theta \neq \theta_0$ and any other asymptotically equivalent test is AUMPI($\alpha$).
\end{theo}

\subsection{Proof of Theorem \ref{Theo_testAUMPI_AR(p)_P3}}\label{Preuve_6}
\begin{proof}

We have just to show that the test $\widetilde{\phi}_n$ is asymptotically equivalent to the test given in Theorem \ref{Theo_AUMPI_test_P3} under $\mathbb{P}_{\theta_0}^{(n)}$ (the null hypothesis). We have,
\begin{equation}\label{moment1_test}
\vert \widetilde{\phi}_n - \phi_n \vert = \widetilde{\phi}_n + \phi_n -2\widetilde{\phi}_n\phi_n.
\end{equation}
Denoting $\delta_n = \langle I(\theta_0)^{-1}\frac{M_n}{\sqrt{n}},\frac{M_n}{\sqrt{n}}\rangle-2\log\left( \frac{\mathcal{L}(\widehat{\theta}_n,X^{(n)})}{\mathcal{L}(\theta_0,X^{(n)})}\right),$ we have for all $\varepsilon>0,$
\begin{equation}\label{moment1_p1}
\vert \widetilde{\phi}_n - \phi_n \vert \leqslant \widetilde{\phi}_n + \phi_n -2\mathbbm{1}_{\left\lbrace 2\log\left( \frac{\mathcal{L}(\widehat{\theta}_n,X^{(n)})}{\mathcal{L}(\theta_0,X^{(n)})}\right) \geqslant C_{\alpha} + \varepsilon \right\rbrace} \mathbbm{1}_{\left\lbrace \delta_n \geqslant - \varepsilon \right\rbrace}.
\end{equation}
It's easy to see that ${\delta_n}\xrightarrow[n\to\infty]{\mathcal{L}}0$ under $\mathbb{P}_{\theta_0}^{(n)},$ then Slutsky's Theorem together with Lemma \ref{Tech4.1} and the continuous mapping Theorem's (see Theorem 2.1 in \cite{billingsley99}) gives, 
\begin{equation}\label{moment1_p2}
{2\mathbbm{1}_{\left\lbrace 2\log\left( \frac{\mathcal{L}(\widehat{\theta}_n,X^{(n)})}{\mathcal{L}(\theta_0,X^{(n)})}\right) \geqslant C_{\alpha} + \varepsilon \right\rbrace} \mathbbm{1}_{\left\lbrace \delta_n \geqslant - \varepsilon \right\rbrace}}\xrightarrow[n\to\infty]{\mathcal{L}}{2\mathbbm{1}_{\left\lbrace Z \geqslant C_{\alpha} + \varepsilon \right\rbrace}}
\end{equation}
under $\mathbb{P}_{\theta_0}^{(n)}$ where $Z \sim \chi_{p}^2.$ Since the second member of (\ref{moment1_p1}) is uniformly bounded with respect to $n$ by $1$, we have,
\begin{equation}\label{moment1_p3}
\limsup\limits_{n \to \infty} \mathbb{E}_{\theta_0}^{(n)}(\vert \widetilde{\phi}_n - \phi_n \vert) \leqslant 2\mathbb{P}(Z \geqslant C_{\alpha}) - 2\mathbb{P}(Z \geqslant C_{\alpha} + \varepsilon),
\end{equation}
and the previous inequality leads to
\begin{equation}\label{moment1_p4}
{\mathbb{E}_{\theta_0}^{(n)}(\vert \widetilde{\phi}_n - \phi_n \vert)}\xrightarrow[n\to\infty]{}0.
\end{equation}
The last condition ensures that the tests $\widetilde{\phi}_n$ and $\phi_n$ are asymptotically equivalent. 
\end{proof}

\nocite{*}

\bibliographystyle{plain}
\bibliography{bibart}

\vspace{10pt}

\end{document}